\documentclass{article}

\usepackage{geometry}
\usepackage[dvipsnames]{xcolor}
\definecolor{myred}{HTML}{c20014}
\definecolor{mygreen}{HTML}{008000}

\usepackage{amsmath}
\usepackage{amsfonts}
\usepackage{array,multirow,makecell}
\usepackage{amsthm}
\usepackage{amssymb}

\usepackage{verbatim}
\usepackage{dsfont}

\usepackage{enumitem}

\usepackage{cite}
\usepackage{mathtools}
\usepackage{sectsty}

\usepackage{hyphenat}

\usepackage{url}
\usepackage[colorlinks=true,linkcolor=myred,citecolor=mygreen]{hyperref}
\hypersetup{linktocpage}
\usepackage[capitalise]{cleveref}

\crefname{equation}{}{}
\Crefname{equation}{Equation}{Equations}

\newtheorem{theo}{Theorem}

\newtheorem{coro}{Corollary}

\theoremstyle{definition}

\theoremstyle{remark}
\newtheorem{rem}{Remark}
\newtheorem{as}{Assumption}
\crefname{as}{Assumption}{Assumptions}
\Crefname{as}{Assumption}{Assumptions}
\crefname{claim}{Claim}{Claims}
\crefname{theo}{Theorem}{Theorems}

\newcommand{\M}{\mathcal{M}}
\newcommand{\dt}[1]{\frac{\mathrm{d}#1}{\mathrm{d}t}}
\newcommand{\dl}[2]{\frac{\partial#1}{\partial#2}}
\newcommand{\ddl}[2]{\frac{\partial^2#1}{\partial#2^2}}
\newcommand{\dddl}[2]{\frac{\partial^3#1}{\partial#2^3}}
\newcommand{\ddddl}[2]{\frac{\partial^4#1}{\partial#2^4}}

\newcommand{\D}{\mathcal{D}}

\newcommand{\B}{\mathcal{B}}
\renewcommand{\d}{\mathrm{d}}

\renewcommand{\L}{\mathcal{L}}

\newcommand{\R}{\mathbb{R}}
\newcommand{\T}{\mathcal{T}}

\newcommand{\C}{\mathcal{C}}
\newcommand{\Op}{\mathcal{O}}

\newcommand{\yr}{y_\mathrm{ref}}

\DeclareMathOperator{\dist}{dist}

\DeclareMathOperator{\id}{id}

\DeclareMathOperator{\range}{Range}

\newcommand{\F}{\mathcal{F}}

\renewcommand{\leq}{\leqslant}
\renewcommand{\geq}{\geqslant}
\renewcommand{\epsilon}{\varepsilon}
\newcommand{\eps}{\varepsilon}

\date{}

\title{
On forwarding techniques for stabilization and set-point output regulation of semilinear infinite-dimensional systems\thanks{This work has been partially supported by the Academy of Finland Grant number 349002.}
}

\author{Nicolas Vanspranghe\thanks{Mathematics and Statistics, Faculty of Information Technology and Communication Sciences, Tampere University, P.O. Box 692, 33101 Tampere, Finland. Emails: \nolinkurl{name.surname@tuni.fi}.}%
\and Lucas Brivadis\thanks{Universit\'e Paris-Saclay, CNRS, CentraleSup\'elec, Laboratoire des Signaux et Syst\`emes, 91190, Gif-sur-Yvette, France. Email: \nolinkurl{lucas.brivadis@centralesupelec.fr}.}
\and Lassi Paunonen\footnotemark[2]
}

\begin{document}

\maketitle

\begin{abstract}
A  stabilizer based on the forwarding technique is proposed for semilinear infinite\hyp{}dimensional systems in cascade form. Sufficient conditions for local exponentially stability and global asymptotic stability of the closed-loop are derived. Results for the problem of local set-point output regulation are also obtained. Finally, an application to a system consisting of a flexible beam attached to a rotating joint is proposed.
\end{abstract}

\section{Introduction}

The stabilization of infinite-dimensional systems presenting a cascade form is a modern challenge in control theory and has been investigated by many researchers in recent years.
In particular, a strong focus has been made on cascades of partial differential equations (PDEs) with ordinary differential equations (ODEs) \cite{MarAst22conf, BalMar23, DEUTSCHER2021109338}, or of ODEs with PDEs \cite{marx2021forwarding, astolfi:hal-04007322}, or of PDEs with other PDEs \cite{auriol2, auriol3}.
Approaches known as ``backstepping'' or ``forwarding'' (each relying on a different paradigm, and adapted to different cascade structures) have been developed to tackle the issue.
In this paper, we focus on the cascade of two infinite-dimensional systems where the first one is semilinear and control-affine,
and the second one is linear, neutrally stable (Lyapunov stable but not asymptotically stable) and driven by a semilinear output of the first one.
This type of systems arises in two contexts in practice: in the stabilization of cascade systems where the second system is linear and driven by an output of the first one, and in the output regulation by means of integral action of semilinear infinite-dimensional systems with 
potentially semilinear output.

Concerning the output regulation problem, many recent works have proposed to extend the linear finite-dimensional theory of \cite{1101104} to linear infinite-dimensional systems by means of an infinite-dimensional internal model principle \cite{Pohjolainen, bymes2000output,rebarber2003internal,paunonen2010internal,paunonen2015controller,doi:10.1137/17M1136407}.
The extension of these works to the infinite-dimensional abstract nonlinear context is still an open problem, although some recent progress have been made in specific contexts
\cite{logemann2000time, natarajan2016approximate}.

We propose to tackle the cascade stabilization problem by means of the forwarding approach that was originally developed for cascades of nonlinear ODEs in \cite{MazPra96, PraOrt01}. This approach is well-suited for the class of cascades under consideration in the present paper, making use of the internal stability of the first subsystem. It could also allow to extend our results to systems with saturated input as in \cite{marx2021forwarding}, and does not rely on any specific hyperbolic structure of the PDE (systems are considered in an abstract framework).
The extension of this strategy to infinite-dimensional systems is an active research area, especially for linear systems (or for linear systems with saturated or cone-bounded nonlinearities applied to the input) \cite{terrand2019adding, MarAst22conf, marx2021forwarding,Nat21}.
In order to deal with nonlinear infinite-dimensional systems, the theory developed in \cite{mattia-forwarding} for finite-dimensional systems by means of incremental stability is instrumental, and has recently been adapted in the context of infinite-dimensional output regulation in \cite{VanBri23}.

The present paper is a continuation of the work proposed in \cite{VanBri23}: while the class of systems under consideration is less generic (here, we restrict ourselves to semilinear systems),
the results obtained here are stronger than the ones presented in \cite{VanBri23} regarding the following points: the input-to-state stability assumption on the first system is weaker; the input acts via a linear operator that may depend on the state; the output of the first system may be nonlinear; the dynamics of the first system is not assumed to be globally contracting; the dynamics of the second system is more general than a pure integrator.
These extensions require modifying the controller designed by forwarding, and in particular improving the result of existence of an invariant graph.
Finally, we apply our results to a system which consists of a flexible beam attached to a rotating joint.

\paragraph*{Organization of the paper}
In \cref{sec:stab}, we tackle the cascade stabilization problem. The   controller we propose ensures local exponential stability and global asymptotic stability under the assumption of existence of some invariant graph. Sufficient conditions for this
are derived in \cref{sec:syl}. We apply our results in the context of output regulation in \cref{sec:reg}, and to an example of flexible structure in \cref{sec:flex}.

\paragraph*{Notation} The norm of a given normed vector space $E$ is denoted by $\|\cdot \|_E$. If $x \in E$ and $r > 0$, $\B_E(x, r)$ denotes the open ball of $E$ centered at $x$ and of radius $r$. 
Let $E_1$ and $E_2$ be normed vector spaces. Then, $\L(E_1, E_2)$ denotes the space of bounded (i.e., continuous) linear operators from $E_1$ to $E_2$ equipped with the operator norm. We say that a map $f : E_1 \to E_2$ is Fr\'echet differentiable at $x \in E_1$ if there exists a (necessarily unique) linear map $\d f(x) \in \L(E_1, E_2)$ such that $f(x + h) = f(x) + \d f(x) h + o(\|h\|_{E_1})$ as $\|h\|_{E_1} \to 0$. Also, $\C^1(E_1, E_2)$ denotes the space of continuously Fr\'echet differentiable maps from $E_1$ to $E_2$, i.e., those $f$ for which the Fr\'echet differential $\d f$ is continuous from $E_1$ to $\L(E_1, E_2)$. By a \emph{locally} Lipschitz continuous map $f : E_1 \to E_2$, we mean a map that is Lipschitz continuous on every bounded subset of $E_1$.  The scalar product of a given Hilbert space $E$ is written $\langle \cdot, \cdot \rangle_E$. If $E_1$ and $E_2$ are Hilbert spaces, each operator $L \in \L(E_1, E_2)$ possesses an adjoint $L^* \in \L(E_2, E_1)$
that is uniquely defined by $\langle Lx, y \rangle_{E_2} = \langle x, L^*y \rangle_{E_1}$ for all $x \in E_1$ and $y \in E_2$. Integrals of functions taking values in Banach spaces are understood in the sense of Bochner.


\section{Stabilization of cascade systems}
\label{sec:stab}


Let $A : \D(A) \to X$ be the infinitesimal generator of a strongly continuous semigroup $\{e^{tA} \}_{t \geq 0}$ on a real Hilbert space $X$. Its domain $\D(A)$ is equipped with the graph norm.
Given input and output spaces $U$ and $Y$, both of which are assumed to be real Hilbert spaces as well, consider the  control system

\begin{subequations}
\label{eq:cascade-x}
\begin{align}
\label{eq:open-loop-x}
&\dot{x} = Ax + f(x) + g(x)u, \\
\label{eq:z-sub}
&\dot{z} = Sz + Cx + h(x),
\end{align}
\end{subequations}
where:
\begin{itemize}
\item $f:X\to X$ is locally Lipschitz continuous, $f(0) = 0$;
\item $g : X \to \L(U, X)$ is locally Lipschitz continuous;
\item $h : X \to Y$ is locally Lipschitz continuous, $h(0) = 0$;
\item $C \in \L(\D(A),  Y)$, i.e., $C$ is $A$-bounded;
\item $S \in \L(Y)$ is skew-adjoint, i.e., $S^* = - S$.
\end{itemize}
The controlled $x$-subsystem is governed by a semilinear equation and has a (possibly nonlinear) output that is fed to the linear $z$-subsystem, which we wish to stabilize. We further assume that the $x$-component is stable or can be pre-stabilized in a way that ensures the following properties.

\begin{as}[Semiglobal Input-to-State Stability (ISS) of the $x$-subsystem]
\label{as:ISS}There exists a Lyapunov functional $V \in \C^1(X, \R)$ 
that is quadratic-like, i.e., there exist $m_1, m_2 > 0$ such that
\begin{equation}
\label{eq:strict-lyap}
m_1 \|x\|^2_X \leq V(x) \leq m_2 \|x\|^2_X, \quad \forall x \in X,
\end{equation}
and there exists $\beta > 0$ such that for all $x \in \D(A)$ and all $u \in U$,
\begin{equation}
\label{eq:ISS0}
\d V(x) [Ax + f(x) + g(x)u] \leq \beta \|u\|^2_U
\end{equation}
Moreover, for any bounded open set $\B\subset X$, there exist a quadratic-like Lyapunov functional $V_\B \in \C^1(X, \R)$ 
and $\alpha_\B, \beta_\B > 0$ such that for all $x \in \D(A)\cap \B$ and all $u \in U$,
\begin{equation}
\label{eq:ISS}
\d V_\B(x) [Ax + f(x) + g(x)u] \leq - \alpha_\B V_\B + \beta_\B \|u\|^2_U.
\end{equation}
\end{as}



\Cref{as:ISS} implies the following formal statements: along all solutions to \cref{eq:open-loop-x},
\begin{equation}
\dot{V} \leq \beta \| u \|^2_U,
\end{equation}
and, given an open bounded set $\B$, if $x$ remains in $\B$ then
\begin{equation}
\dot{V}_\B \leq - \alpha_\B V_\B + \beta_\B \|u\|^2_U,
\end{equation}
for some Lyapunov function $V_\B$. While this formulation may seem  unusual, it will naturally arise in our PDE application.

The next assumption connects the $x$- and $z$-subsystems and is instrumental in building a forwarding-based controller.
\begin{as}[Invariant graph]
\label{as:forwarding}
There exists a map $\M \in \C^1(X, Y)$ with $\d \M$ locally Lipschitz continuous such that $\M(0) = 0$ and for all $x \in \D(A)$,
\begin{equation}
\label{eq:forwarding}
\d \M(x) (A + f)(x)  = S \M(x) +  (C + h)(x).
\end{equation}
\end{as}
Its geometric interpretation  is that, when $u = 0$, the graph of $\M$ is an invariant manifold for the  cascade \cref{eq:cascade-x}. Under \cref{as:forwarding}, we consider the nonlinear state feedback
\begin{equation}
\label{eq:forwarding-feedback}
u = g(x)^*\d \M(x)^* [z - \M(x)],
\end{equation}
which can be modelled as a locally Lipschitz map on the extended state space $X \times Y$. 
\begin{theo}[Well-posedness]
\label{claim:wp}xz
For any initial data $[x_0, z_0]\in X \times Y$, there exists a unique (global) mild solution $[x, z] \in \C(\mathbb{R}^+, X \times Y)$ to the closed-loop equations \cref{eq:cascade-x}-\cref{eq:forwarding-feedback}. If $[x_0, z_0] \in \D(A) \times Y$, then $[x, z]$ is a classical solution and enjoys the regularity $[x, z] \in \C(\R^+, \D(A) \times Y) \cap \C^1(\R^+, X \times Y)$. Furthermore, for any $\tau > 0$, the map $[x_0, z_0] \mapsto [x, y]$ is continuous from $X \times Y$ to $\C([0, \tau], X \times Y)$ equipped with the uniform norm.
\end{theo}

\begin{proof}
Observe that \cref{eq:cascade-x} in closed-loop with \cref{eq:forwarding-feedback} constitute a locally Lipschitz perturbation of the linear equations $\dot{x} = Ax, \dot{z} = Cx$. Those
 generate a strongly continuous semigroup on $X \times Y$. Indeed, taking advantage of the $A$-boundedness of $C$, one can show txhat the operator matrix
$
\begin{bsmallmatrix}
A & 0 \\
C & 0
\end{bsmallmatrix}
$
with domain $\D(A) \times Y$  is closed and  has nonempty resolvent set. In that case, semigroup generation is equivalent to existence of unique classical solutions for all initial data in the domain \cite[Theorem 3.1.12]{AreBat01book}, which is immediate.
Then, existence and uniqueness of local mild and classical solutions for the nonlinear problem together with uniform continuous dependence on the initial data follow from \cite[Theorem 11.1.5]{CurZwa20book}. The proof that all closed-loop solutions are global is postponed: our stability analysis will show that they cannot blow up in finite time -- see \cref{eq:dec-W} below.x
\end{proof}
\begin{rem}
All formal computation performed in the sequel can be justified 
by considering classical solutions and passing to the limit in suitable expressions for general initial data.
\end{rem}

\begin{theo}[Local Exponential Stability]
\label{claim:LES}Suppose that
\begin{equation}
\label{eq:range-cond}
\range \d \M(0) g(0) = Y.
\end{equation}
Then the zero equilibrium is Locally Exponentially Stable 
for \cref{eq:cascade-x} in closed loop with \cref{eq:forwarding-feedback}.
\end{theo}

\begin{proof}
Consider the Lyapunov candidate
\begin{equation}
W(x, z) \triangleq V(x) + \frac{\beta}{4} \|z - \M(x) \|^2_Y
\end{equation}
with $\beta$ as in \cref{eq:ISS0}. Along solutions to \cref{eq:cascade-x}-\cref{eq:forwarding-feedback}, we have
\begin{equation}
\label{eq:derivative-W}
\dot{W} = \dot{V} + \frac{\beta}{2} \langle z - \M(x),  \dot{z} - \d \M(x) \dot{x} \rangle_Y.
\end{equation}
Since $S$ is skew-adjoint, 
\begin{equation}
\langle z - \M(x), Sz \rangle_Y = \langle z - \M(x), S \M(x) \rangle_Y
\end{equation}
Therefore, plugging \cref{eq:open-loop-x,eq:forwarding-feedback,eq:forwarding} leads to
\begin{equation}
\begin{aligned}
\langle z - \M(x),  \dot{z} - \d \M(x) \dot{x} \rangle_U  &= \langle z - \M(x), - \d \M(x)g(x) u \rangle_U  \\ &= - \|g(x)^* \d \M(x)^*[z - \M(x)]\|^2_U.
\end{aligned}
\end{equation}
Using the (global) ISS property \cref{eq:ISS0}, we obtain
\begin{equation}
\label{eq:dec-W}
\dot{W} \leq - \frac{\beta}{2} \| u \|^2_U \leq 0.
\end{equation}
\Cref{eq:dec-W} shows that the $W$-sublevel sets
\begin{equation}
\mathcal{N}_c \triangleq \{ [x_0, z_0] \in X \times Y : W(x_0, z_0) \leq c \}, \quad c > 0,
 \end{equation}
are positively invariant under the flow of \cref{eq:cascade-x}-\cref{eq:forwarding-feedback}. Furthermore, since $\M$ is continuous and vanishes at zero, using also \cref{eq:strict-lyap}, for any $\eps > 0$ we can find $c >0$ such that $\mathcal{N}_c \subset \B_{X \times Y}(0, \eps)$ and, conversely, for any $c > 0$, there exists $\delta > 0$ such that $ \B_{X\times Y}(0, \delta) \subset \mathcal{N}_c$. This combined with positive invariance of each $\mathcal{N}_c$ proves that the origin is Lyapunov stable.
On the other hand, because of the range condition \cref{eq:range-cond}, a transposition argument -- see, e.g., \cite[Theorem 2.20]{Bre11book} -- provides $\lambda > 0$ such that
\begin{equation}
\|g(0)^* \d \M(0)^* y \|^2_U \geq \lambda \|y\|^2_Y, \quad \forall y \in Y.
\end{equation}
Because $x \mapsto g(x) \d \M(x)$ is continuous, it is possible to find some ball $\Op \triangleq \B_X(0, r)$, $r > 0$, such that
\begin{equation}
\label{eq:local-coerc}
\|g(x)^* \d \M(x)^* y \|^2_U \geq \frac{\lambda}{2} \|y\|^2_Y, \quad \forall x \in \Op, \forall y \in Y.
\end{equation}
We are now ready to prove local exponential stability of the origin. 
Fix $c >0$ such that
$\mathcal{N}_c \subset \Op \times Y
$ and pick $\delta > 0$ such that $\B_{X \times Y}(0, \delta) \subset \mathcal{N}_c$. The
$x$-coordinate of any closed-loop solution originating from $\mathcal{N}_c$ remains in the $V$-sublevel set $ \B \triangleq \{ x \in X : V(x) < 2c \}$, which is bounded. Thus, by \cref{as:ISS} there exists a Lyapunov functional $V_\B$ and positive constants $\alpha_\B$, $\beta_\B$ such that $\dot{V}_\B \leq -\alpha_\B V_\B + \beta_\B \|u\|^2_U$ along all closed-loop solutions originating from $\mathcal{N}_c$.
We let
\begin{equation}
\label{eq:Wb}
W_\B(x, z) \triangleq V_\B(x) + \frac{\beta_\B}{4} \|z - \M(x) \|^2_Y
\end{equation}
and this time,  similarly as in \crefrange{eq:derivative-W}{eq:dec-W} but taking advantage of \cref{eq:local-coerc}, we obtain
\begin{equation}
\label{eq:Wb-strict}
\begin{aligned}
\dot{W}_\B &
 \leq - \alpha_\B {V}_\B - \frac{\beta_\B \lambda}{4} \|z - \M(x) \|^2_Y \\
 & \leq - \min  \{ \alpha_\B, \lambda  \} W_\B
\end{aligned}
\end{equation}
along all solutions to \cref{eq:cascade-x}-\cref{eq:forwarding-feedback} originating from $\mathcal{N}_c$.
It then follows from \cref{eq:Wb-strict} and Gr\"onwall's lemma that 
\begin{equation}
W_\B(x(t), z(t)) \leq e^{- \min  \{ \alpha_\B, \lambda  \} t} W_\B(x_0, z_0)
\end{equation}
for all $t \geq 0$ and any closed-loop solution $[x, z]$  with initial data $[x_0, z_0]$ taken in $\mathcal{N}_c$. To conclude,
we recall that $\M$ is Lipschitz continuous on the bounded set $\B$, $\M(0) = 0$ and $0 \in \B$. On the other hand, $V_\B$ has quadratic upper and lower bounds. Thus, using a couple of triangular inequalities, we deduce that there exist positive constants $K_1, K_2$ such that $K_1 W_\B(x, z) \leq \|x\|^2_X + \|z\|^2_Y \leq K_2 W_\B(x, z)$ for all $x \in \B$ and $z \in Y$, which completes the proof.
\end{proof}

\begin{theo}[Global Asymptotic Stability]
\label{claim:GAS}
In addition to \cref{eq:range-cond}, suppose that
\begin{enumerate}[label=(\roman*)]
\item 
\label{it:L2-adm}
The linear semigroup $\{e^{tA}\}_{t \geq 0}$ is exponentially stable;
\item $Y$ is finite-dimensional.
\end{enumerate}
Then the zero equilibrium is Globally Asymptotically Stable for \cref{eq:cascade-x} in closed loop with \cref{eq:forwarding-feedback}.
\end{theo}
\begin{proof}
Let the initial data $[x_0, z_0] \in X \times Y$ be fixed. If $W(x_0, z_0) = 0$, then $x_0 = 0$ and $z_0 = 0$. Otherwise, it follows again from \cref{eq:dec-W} that the $x$-coordinate of the closed-loop solution originating from $[x_0, z_0]$ remains in an open bounded set $\B$ of the form $\B = \{ x \in X : V(x) < 2W(x_0, z_0) \}$. Thus, by \cref{as:ISS}, there exists a Lyapunov functional $V_\B$ and positive constants $\alpha_\B, \beta_\B$ such that $W_\B$ constructed just as in \cref{eq:Wb} satisfies
\begin{equation}
\label{eq:Wb-bis}
\dot{W}_\B \leq - \alpha_\B V_\B - \frac{\beta_\B}{2} \|u\|^2_U
\end{equation}
along the closed-loop solution originating from $[x_0, z_0]$. Integrating \cref{eq:Wb-bis} over $(0, + \infty)$ yields
\begin{equation}
\label{eq:Wb-integrated}
\int_0^{+\infty} \alpha_\B V_\B(x(t)) + \frac{\beta_\B}{2} \|u(t)\|^2_U \, \d t \leq 2 W_\B(x_0, z_0).
\end{equation}
In particular,  $x \in L^2(0, + \infty; X)$ and $u \in L^2(0, + \infty; U)$. Because $x$ is bounded in $X$ and $f$ is locally Lipschitz continuous, this also implies that $f(x) \in L^2(0, + \infty; X)$. On the other hand, because $g$ is locally Lipschitz continuous as well and $u \in L^2(0, + \infty; U)$, we must have $g(x)u \in L^2(0, +\infty; X)$. Let $\Xi \triangleq f(x) + g(x) u$. Then $\Xi \in L^2(0, + \infty; X)$ and $x$ solves the Cauchy problem
$
\dot{x} = Ax + \Xi$, $x(0) = x_0$.
Since $\{e^{tA}\}_{t \geq 0}$ is exponentially stable, it follows from \cite[Lemma 5.2.2]{CurZwa20book} that $x(t) \to 0$ in $X$ as $t \to + \infty$.
Another consequence of \cref{eq:dec-W} is that $z$ remains bounded in $Y$, which is assumed to be finite-dimensional; hence relative compactness of $\{z(t), t \geq 0 \}$ in $Y$. We are now in the position to carry out a standard LaSalle invariance argument \cite{Las60,Daf78}. Indeed,
we now know that 
the $\omega$-limit set
$
\omega(x_0, z_0)$  of the singleton $[x_0, z_0]$ under the flow of \cref{eq:cascade-x}-\cref{eq:forwarding-feedback} is nonempty; it is in fact included in $\{0\} \times Y$ and thus compact, implying
that, as $t \to + \infty$,
\begin{equation}
\label{eq:omega-att}
\dist( [x(t), z(t)], \omega(x_0, z_0)) \to 0.
\end{equation}
Given $[\tilde{x}_0, \tilde{z}_0] \in \omega(x_0, z_0)$, let $[\tilde{x}, \tilde{z}]$ be the closed-loop solution originating from $[\tilde{x}_0, \tilde{z}_0]$. Since the $W$-sublevel sets are all closed and invariant under \cref{eq:cascade-x}-\cref{eq:forwarding-feedback}, $V(\tilde{x}) \leq W(x_0, z_0) < 2 W(x_0, z_0)$ for all $t\geq 0$; thus,  \cref{eq:Wb-bis} is valid for $[\tilde{x}, \tilde{z}]$ as well.
By (strict) invariance of $\omega(x_0, z_0)$, the continuous, monotone decreasing and lower-bounded function $t \mapsto W_\B(\tilde{x}(t), \tilde{z}(t))$ must in fact be constant. We then infer from the $[\tilde{x}, \tilde{z}]$-version of \cref{eq:Wb-bis} that 
\begin{equation}
\label{eq:lassalle-zero}
\tilde{x}(t) = 0, \quad \tilde{u}(t) = 0, \quad \forall t \geq 0,
\end{equation}
where $\tilde{u}$ is the control for $[\tilde{x}, \tilde{z}]$. \Cref{eq:lassalle-zero} leads to
\begin{equation}
g(0)^* \d \M(0)^* \tilde{z}(t) = 0, \quad t \geq 0,
\end{equation}
where we recall that $\M(0) = 0$. Because $ \d \M(0)g(0)$ is assumed to be surjective, $g(0)^* \d \M(0)^*$ is injective and we finally obtain
$\omega(x_0, z_0) = \{0\}
$. By \cref{eq:omega-att}, this shows  that $0$ is globally attractive for the closed-loop \cref{eq:cascade-x}-\cref{eq:forwarding-feedback}. 
\end{proof}



\section{Solving the nonlinear Sylvester equation}\label{sec:syl}

In this section, we establish sufficient conditions under which a suitable solution $\M$ to \cref{eq:forwarding} exists. Consider the uncontrolled $x$-equation:
\begin{equation}
\label{eq:unc-x}
\dot{x} = Ax + f(x).
\end{equation}
It is clear from the arguments in the proof of \cref{claim:wp} that \cref{eq:unc-x} gives rise to a dynamical system in $X$, with similar regularity and uniform approximation properties. We denote by $\{\T_t\}_{t \geq 0}$ the associated evolution semigroup: $t \mapsto \T_t x_0$ is the unique solution $x$ to \cref{eq:unc-x} with initial condition $x(0) = x_0$. Let us now discuss the consequences of \cref{as:ISS} (with $u = 0$) on the stability of \cref{eq:unc-x}. The positive invariance of each $V$-sublevel set under the flow of \cref{eq:unc-x} combined with the existence of a strict Lyapunov functional $V_\B$ on each bounded set $\B$ of $X$ imply \emph{semiglobal} exponential stability: for each bounded set $\B$ there exist $M_\B\geq 1$ and $\mu_\B > 0$ such that
\begin{equation}
\label{eq:stab-T}
\|\T_t x_0 \|_X \leq M_{\B} e^{- \mu_\B t} \|x_0\|_X, \quad \forall x_0 \in \B, \forall t \geq 0.
\end{equation}
In particular, the zero equilibrium \emph{uniformly} attracts the bounded sets of $X$, i.e.,  for any $\eps$-ball, $\eps > 0$, around the origin and any bounded set $\B$, there exists a time $T$ after which all solutions originating from $\B$ remain in that $\eps$-ball.
The following additional assumptions are now in force.

\begin{as}
\label{as:smooth}
The maps $f$ are Fr\'echet differentiable with locally Lipschitz continuous differentials. Also, without loss of generality, $\d f(0) = 0$ and $\d h(0) = 0$.
\end{as}

\begin{as}
\label{as:first-var}
There exist a coercive self-adjoint operator $P \in \L(X)$ and a positive constant $\mu$ such that 
\begin{equation}
\label{eq:ineq-first-var}
\langle Ax, P x \rangle_X \leq - \mu \|x\|^2_X, \quad \forall x \in \D(A).
\end{equation}
\end{as}
\Cref{as:first-var} implies that  $\{e^{tA}\}_{t \geq 0}$ 
is exponentially stable with 
a coercive quadratic Lyapunov functional. 
\begin{theo}[Existence of $\M$]
\label{claim:sylvester}
Let $\M_0 \in \L(X, Y)$ be the (unique) solution to the linear Sylvester equation
\begin{equation}
\label{eq:sylvester-lin}
\M_0 A = S \M_0 + C.
\end{equation}
Then, the unique solution $\M$ to \cref{eq:forwarding}, $\M(0) = 0$, is given by
\begin{equation}
\label{eq:M-nonlin}
\M(x) = \M_0x + \int_0^{+ \infty} e^{-t S} [ \M_0 f(\T_t x) - h(\T_t x)] \, \d t.
\end{equation}
Furthermore, $\M \in \C^1(X, Y)$ and $\d \M$ is locally Lipschitz continuous.
\end{theo}

\begin{proof}
First, 
since $\{e^{tA}\}_{t \geq 0}$ is exponentially stable, $0$ lies in the resolvent set of $A$.  On the other hand, $S$ is skew-adjoint and bounded, and thus generates a (uniformly continuous) group $\{e^{tS}\}_{t \in \R}$ of isometries on $Y$. 
With that in mind, it can be checked by following the proof of \cite[Lemma III.4]{NatGil14} or using \cite[Theorem 2.1]{Ngu01} that the unique solution $\M_0 \in \L(X, Y)$ to \cref{eq:sylvester-lin} is given by
\begin{equation}
\M_0x = CA^{-1}x - \int_0^{+ \infty} S e^{-tS} C A^{-1}e^{tA}x \, \d t
\end{equation}
for all $x \in X$. 
Now we look for a Fr\'echet differentiable solution $\M$ to \cref{eq:forwarding} of the form $\M(x) = \M_0 x + \F(x)$. Since $\M_0$ solves \cref{eq:sylvester-lin}, such a map $\M$ satisfies \cref{eq:forwarding} if and only if
\begin{equation}
\M_0 f(x) + \d \F(x) (A + f)(x) = S \F(x) + h(x)
\end{equation}
for all $x \in \D(A)$, or equivalently,
\begin{equation}
\label{eq:calcul-dF-T}
\d \F(\T_t x) \dt{} \T_t x - S \F (\T_t x)  = - \M_0 f(\T_t x) +  h(\T_tx), \quad \forall x \in \D(A), \forall t \geq 0.
\end{equation}
By applying (the invertible operator) $e^{-tS}$ to \cref{eq:calcul-dF-T}, we see that $\M$ solves \cref{eq:forwarding} if and only if for all $x \in \D(A)$ and $t \geq 0$,
\begin{equation}
\label{eq:calcul-dF-final}
\dt{} e^{-t S} \F(\T_t x) = - e^{-t S} [ \M_0 f(\T_t x) -  h(\T_t x)].
\end{equation}
Since $\M$ (and thus $\F$) is assumed to be continuous and vanish at $0$, we can integrate \cref{eq:calcul-dF-final} and obtain that if $\M$ is indeed a solution to \cref{eq:forwarding}, it must be given by \cref{eq:M-nonlin}. Indeed, recall here that $\M_0$ is unique and the integral in \cref{eq:M-nonlin} is absolutely convergent because of \cref{eq:stab-T} together with the property that $f$ and $h$ are linearly bounded on bounded sets.
Conversely, we have to prove that the map $\M$ defined by \cref{eq:M-nonlin} for any $x \in X$ is Fr\'echet differentiable and solves \cref{eq:forwarding} or, equivalently, satisfies \cref{eq:calcul-dF-final} for all $x \in \D(A)$ and $t \geq 0$, where we let $\F \triangleq \M - \M_0$. Assume for the moment that $\M$ is Fr\'echet differentiable and let $x \in \D(A)$. 
For all $t \geq 0$,
\begin{equation}
\label{eq:diff-F}
e^{-t S} \F(\T_t x) - \F(x)  = -\int_0^{t} e^{-s S} [ \M_0 f(\T_s x) -  h(\T_s x)] \, \d s.
\end{equation}
Dividing \cref{eq:diff-F} by $t > 0$ and letting $t \to 0$ yield
\begin{equation}
\dt{} e^{-tS} \F(\T_t x) \bigg|_{t = 0} = h(x) - \M_0 f( x),
\end{equation}
which implies \cref{eq:calcul-dF-T}. The differentiability of $\M$ (along with the Lipschitz continuity of the differential) is a   extension of the (lengthy) proof of \cite[Theorem 3.4]{VanBri23}. It is omitted here.\footnote{See \cref{sec:add-proof} for a sketch of the proof.}
In the sequel, we shall use that $\d \M(0) = \M_0$, which follows from \cref{eq:M-nonlin}.
\end{proof}


We now know that \cref{as:smooth,as:first-var} imply \cref{as:forwarding} and the solution $\M$ to \cref{eq:forwarding} is unique. As a result, we can reformulate the hypothesis \cref{eq:range-cond} of \cref{claim:LES} in terms of the original control system only and, in the case $S = 0$, recover the classical non-resonance condition \cite{Isi03book}.
\begin{coro}[Non-resonance condition]
The range condition \cref{eq:range-cond} 
reads as
$\range \M_0g(0) = Y$,
where $\M_0$ is the unique solution to the Sylvester equation \cref{eq:sylvester-lin}. In particular, if $S = 0$, \cref{eq:range-cond} reads as follows:
\begin{equation}
\label{eq:non-resonance}
\range CA^{-1}g(0) = Y,
\end{equation}
\end{coro}
\begin{proof}
This is a consequence of the uniqueness of $\M$ and the property that $\d \M(0) = \M_0$.
\end{proof}

\section{Local set-point output regulation}\label{sec:reg}

In this section, we present an application of the forwarding approach for stabilization of cascade systems in the context of set-point output regulation. Let $\yr \in Y$ be a (small) deviation from the output $y = Cx + h(x)$ at the equilibrium (here, the origin). We wish to find a control $u$ steering $y$ to $\yr$ while maintaining $x$ bounded.
In the spirit of \cite{mattia-forwarding,VanBri23}, consider in place of \cref{eq:z-sub}
\begin{equation}
\label{eq:z-sub-int}
\dot{z} = Cx + h(x) - \yr.
\end{equation}
The crucial property of integral action is that at \emph{any} equilibrium, the output must be at the desired value $\yr$.
We assume that \cref{as:ISS,as:smooth,as:first-var} are satisfied, which in turn means that \cref{as:forwarding} is satisfied as well by \cref{claim:sylvester}. In particular, the unique solution $\M$ to \cref{eq:forwarding}, $\M(0) = 0$ is given by \cref{eq:M-nonlin}  -- here, $S = 0$, and thus $\M_0 = CA^{-1}$. We thus consider the nonlinear feedback  \cref{eq:forwarding-feedback} of the state $[x, z]$ governed by \cref{eq:cascade-x}-\cref{eq:z-sub-int}.

\begin{theo}[Set-point output regulation]
\label{th:set-point}
Suppose that
\begin{equation}
\range CA^{-1}g(0) = Y.
\end{equation}
Then, there exists $r > 0$ such that for any $\yr \in \B_Y(0, r)$, the following property holds: the $x$-subsystem \cref{eq:cascade-x} supplemented with the output integrator \cref{eq:z-sub-int} and in closed loop with \cref{eq:forwarding-feedback} possesses an equilibrium $[x^\star, z^\star] \in \D(A) \times Y$ that is Locally Exponentially Stable, satisfies $(C + h)(x^\star) = \yr$ and whose basin of attraction contains the origin.
\end{theo}
\begin{rem}
That the closed-loop system in presence of the additional term $\yr$ is well-posed and forward complete can be verified by following the proofs of \cref{claim:wp,claim:LES}.
\end{rem}
\Cref{th:set-point} can be proved by following the strategy described in \cite[Section 4.2.1]{VanBri23}.\footnote{
It is however quite lenghty, so we only sketch the differences with respect to \cite[Theorem 3.2]{VanBri23} in \cref{sec:add-proof}.}
\begin{rem}
A similar result can be obtained in presence of small constant disturbance $d$ in the $x$-equation \cref{eq:cascade-x}, as considered in \cite{VanBri23}. However, this requires global Lipschitz continuity of $\M$ and $\d \M$, which is much harder to obtain, and is not guaranteed a priori by \cref{claim:sylvester}.
\end{rem}

\section{Applications to flexible structures}\label{sec:flex}

We consider the planar motion of a flexible homogeneous Euler-Bernoulli beam of length $L$ attached to a rotating joint at one end and free at the other end. The deflection $w(\xi, t)$ in the beam frame at the position $\xi\in[0, L]$ and time $t\in\R^+$ and the rotation angle $\theta(t)$ are governed by the following set of equations:
\begin{subequations}
\label{eq:w-beam}
\begin{align}
&\rho \ddl{w}{t} + \lambda \dl{w}{t}   + EI \ddddl{w}{\xi} + \rho \xi \ddot{\theta} - \rho \dot{\theta}^2 w = 0, \\
&I_R \ddot{\theta}(t) = EI \ddl{w}{\xi}(0, t)  + \tau(t),  \\
& w(0, t) = \dl{w}{\xi}(0, t) = 0, \\
&\dddl{w}{\xi}(L, t) = \ddl{w}{\xi} (L, t) = 0, 
\end{align}
\end{subequations}
where $\lambda$ is a viscous damping coefficient, $E$ is the Young modulus, $I$ and $\rho$ are the moment of inertia and the density of the cross section, $I_R$ is the moment of inertia of the rotating joint, and $\tau$ is the torque applied to the joint, which is our control input. We refer the reader to \cite{Mor91} for more details on that model.
For the sake of simplicity, we set all physical constants to $1$, with the exception of $\lambda$.\footnote{
This is to highlight the fact that, although we take advantage of viscous damping to considerably simplify computations, $\lambda > 0$ can be taken small.}

A possible control objective is reference tracking of the angular position $\theta$.
Let $\theta_\mathrm{ref} \in \R$ be the desired angle.
We start with some important observations. If we set the torque input as
\begin{equation}
\label{eq:torque-stat}
\tau = - \theta + \theta_\mathrm{ref} + \tilde{\tau},
\end{equation}
then, in the new coordinate system $[w, \theta] \mapsto [v, \phi]$ where
\begin{equation}
\phi(t) \triangleq \theta(t) - \theta_{\mathrm{ref}}, \quad v(\xi, t) \triangleq w(\xi,t) + \xi \phi(t),
\end{equation}
the equations of motion \cref{eq:w-beam} supplied with \cref{eq:torque-stat} become:
\begin{subequations}
\label{eq:beam-v}
\begin{align}
&\ddl{v}{t} + \lambda \dl{v}{t} + \ddddl{v}{\xi}   - \lambda \xi \dot{\phi}  - \dot{\phi}^2(v -  \xi \phi) = 0, \\
&\ddot{\phi}(t) = \ddl{v}{\xi}(0, t) - \phi(t) + \tilde{\tau}(t), \\
&\dddl{v}{\xi}(L, t) = \ddl{v}{\xi} (L, t) = v(0, t) = 0, \\
& \dl{v}{\xi}(0, t) = \phi(t).
\end{align}
\end{subequations}
Note in particular that the new equations \cref{eq:beam-v} do \emph{not} depend on the choice of $\theta_\mathrm{ref}$. 
Now, in order to pre-stabilize the plant, an energy approach suggests the following nonlinear feedback in torque:
\begin{equation}
\label{eq:torque}
\tilde{\tau} = - \dot{\phi} \int_\Omega v \dl{v}{t} \,  \d \xi + (\phi \dot{\phi}  - \lambda) \int_\Omega \xi \dl{v}{t} \, \d \xi   - \dot{\phi} + u.
\end{equation}
Indeed, this yields
\begin{equation}
\label{eq:energy-balance}
\frac{1}{2} \dt{} \left (  \int_\Omega \left | \dl{v}{t}  \right |^2 + \left | \ddl{v}{\xi}  \right |^2 \d \xi + |\dot{\phi}|^2 + |\phi|^2 \right )   = - \lambda \int_\Omega \left | \dl{v}{t} \right |^2 \d \xi -  |\dot{\phi}|^2 + u \dot{\phi}
\end{equation}
along trajectories of \cref{eq:beam-v} in closed loop with \cref{eq:torque}.
\begin{rem}
The pre-stabilizer given in \cref{eq:torque} is different from the control laws introduced in \cite{Mor91}.
\end{rem}
Let us introduce an operator model for the control system \cref{eq:beam-v}-\cref{eq:torque}. 
Having set $\Omega \triangleq (0, L)$, we define the spaces
\begin{subequations}
\begin{align}
&H_0 \triangleq L^2(\Omega) \times \R, \\
& H_1 \triangleq \left \{  \begin{bmatrix} v \\ \phi \end{bmatrix}\in H^2(\Omega) \times \R  
\left |
\begin{aligned}
&\dl{v}{\xi}(0) = \phi \\
&v(0) = 0
\end{aligned}
\right.
\right \},
\end{align}
\end{subequations}
which we equip with the following scalar products:
\begin{subequations}
\begin{align}
&\left \langle
\begin{bmatrix}
p_1 
\\ \omega_1
\end{bmatrix},
\begin{bmatrix}
p_2 \\ \omega_2
\end{bmatrix} 
\right \rangle_{H_0} \triangleq \omega_1 \omega_2 +  \int_\Omega p_1 p_2 \, \d \xi, \\
\label{eq:scalar-H1}
&\left \langle
\begin{bmatrix}
v_1 
\\ \phi_1
\end{bmatrix},
\begin{bmatrix}
v_2 \\ \phi_2
\end{bmatrix} 
\right \rangle_{H_1} \triangleq \phi_1 \phi_2 + \int_\Omega \ddl{v_1}{\xi}  \ddl{v_2}{\xi} \, \d \xi.
\end{align}
\end{subequations}
Then, $H_0$ and $H_1$ are Hilbert spaces.\footnote{
One can proceed by showing that $H_1$ is a closed subspace of $H^2(\Omega) \times \R$ and therefore a Hilbert space if equipped with the inherited scalar product. That \cref{eq:scalar-H1} defines an equivalent norm is obtained with standard arguments.}
Now, let
$
X \triangleq H_1 \times H_0
$
equipped with its product Hilbertian structure. We define an unbounded operator $A : \D(A) \to X$ by
\begin{subequations}
\begin{align}
&\D(A) \triangleq \left \{ \begin{bmatrix}
v \\ \phi \\ p \\ \omega
\end{bmatrix} \in H_1 \times H_1 \;
\left |\;
\begin{aligned}
& v \in H^4(\Omega), \\
& \dddl{v}{\xi}(L) = 0, \\& \ddl{v}{\xi} (L) = 0
\end{aligned}
\right.
\right \}, \\
& A \begin{bmatrix}
v \\ \phi \\ p \\ \omega
\end{bmatrix} \triangleq
\begin{bmatrix}
p \\ \omega \\
- \ddddl{v}{\xi} -  \lambda p \\
\ddl{v}{\xi}(0) - \omega -  \phi
\end{bmatrix}, \quad   \forall \begin{bmatrix}
v \\ \phi \\ p \\ \omega
\end{bmatrix} \in \D(A).
\end{align}
\end{subequations}
It follows from the Lumer-Phillips theorem and standard arguments that $A$ is the generator of a contraction semigroup on $X$.
Here, as an input space we let $U \triangleq \R$, and
the control map $g$ is linear and given by $g(x)u = Bu = [0, 0, 0, u]$ for all $x \in X$ and $u \in U$.  We then put all the nonlinear terms from \cref{eq:beam-v}-\cref{eq:torque} into a locally Lipschitz map $f : X \to X$ that meets the requirements of \cref{sec:stab} and also \cref{as:smooth}.

Letting
$
V \triangleq (1/2) \| \cdot \|^2_X 
$,
the energy balance \cref{eq:energy-balance} reads as
\begin{equation}
\label{eq:Vdot-beam}
\dot{V} = - \lambda \int_\Omega \left | \dl{v}{t} \right |^2 \d \xi -  |\dot{\phi}|^2 + u \dot{\phi} \leq \frac{1}{2} |u|^2
\end{equation}
along solutions to \cref{eq:beam-v}-\cref{eq:torque}.
Now, for \cref{as:ISS} to be satisfied, we also need strict control Lyapunov functionals on each (open) bounded subset of $X$. 
To that end,
we can ``strictify'' the total energy $V$.
Let $\eps > 0$ and define
\begin{equation}
\label{eq:V-eps-def}
V_\eps(v, \phi, p, \omega) \triangleq V(v, \phi, p,  \omega)  + \eps \int_\Omega v p \, \d \xi + \eps \phi \omega    +  \frac{\eps \lambda}{2} \int_\Omega |v|^2 \, \d \xi + \frac{ \eps}{2} |\phi|^2
\end{equation}
for all $[v, \phi, p, \omega] \in X$. Then, along solutions to \cref{eq:beam-v}-\cref{eq:torque},
\begin{equation}
\label{eq:V-eps}
\dot{V}_\eps = (\eps - \lambda) \int_\Omega \left | \dl{v}{t} \right |^2 \d \xi + (\eps - 1) |\dot{\phi}|^2
  + u \dot{\phi}   - \eps \int_\Omega \left | \ddl{v}{\xi} \right |^2 \, \d \xi - \eps |\phi|^2    + \eps u \phi + \eps R(v, \phi, \dot{v}, \dot{\phi}),
\end{equation}
where the term $R$ is given by:
\begin{equation}
\label{eq:def-R}
R(v, \phi, \dot{v}, \dot{\phi}) \triangleq  - \lambda \dot{\phi} \int_\Omega \xi v \, \d \xi  -  \dot{\phi}^2\int_\Omega (v -  \xi \phi) v \, \d \xi 
 - \phi \dot{\phi} \int_\Omega v \dl{v}{t} \,  \d \xi  +  (\phi \dot{\phi}  - \lambda) \phi  \int_\Omega \xi \dl{v}{t} \, \d \xi.
\end{equation}
Let us deal with the first term in \cref{eq:def-R}: we can estimate
\begin{equation}
\left |  \lambda \dot{\phi} \int_\Omega \xi v \, \d \xi \right |  \leq \frac{1}{4 \eps} |\dot{\phi}|^2 + \eps \lambda^2 L^3  \int_\Omega |v|^2 \, \d \xi.
\end{equation}
Then, we get
\begin{equation}
\label{eq:est-R-1}
\eps \left |  \lambda \dot{\phi} \int_\Omega \xi v \, \d \xi \right |  \leq \frac{1}{4} |\dot{\phi}|^2    + \eps^2 \lambda^2 L^3 K  |\phi |^2   + \eps^2 \lambda^2 L^3  K \int_\Omega \left | \ddl{v}{\xi}  \right |^2 \d \xi, 
\end{equation}
where $K > 0$ is some constant coming from the equivalence of the norms $\|\cdot \|_{H_1}$ and $\|\cdot \|_{H^2(\Omega) \times \R}$ on $H_1$. Thus, going back to \cref{eq:V-eps}, we see that all terms of \cref{eq:est-R-1} can be absorbed into the negative part of the right-hand side of \cref{eq:V-eps} -- here, the $\eps^2$-prefactor is important. Next, let $\B$ be a fixed open bounded subset of $X$. 
We then examine the remainder of \cref{eq:def-R}, which we denote by $R'$, and observe that it can be estimated as follows: there exists a positive constant $K_{\B}$  such that
\begin{equation}
|R'(v, \phi, p, \omega)| \leq K_\B \int_\Omega |p|^2 \, \d \xi + K_\B |\omega|^2
\end{equation}
for all $[v, \phi, p, \omega] \in \B$. Therefore, for any solution to \cref{eq:beam-v}-\cref{eq:torque} that remains in $\B$, the term $R'$, once multiplied by $\eps$, can be absorbed into the ``good'' terms in \cref{eq:V-eps}, provided that $\eps$ is chosen sufficiently small. Finally, we are left with the term $u \dot{\phi} + \eps u \phi$, which is readily dealt with using Young's inequality. At this point we have proved that for each (open) bounded subset $\B$ of $X$, there exists $\eps > 0$ such that the Lyapunov candidate $V_\eps$ as defined in \cref{eq:V-eps-def} satisfies the requirements of \cref{as:ISS}.
Note that with any $V_\eps$ with $\eps$ sufficiently small, we can also prove that the semigroup generated by $A$ (i.e., the linear part of \cref{eq:beam-v}-\cref{eq:torque}) is exponentially stable and possesses a coercive Lyapunov functional. Indeed, because the pre-stabilizing feedback \cref{eq:torque} is meant to cancel out the nonlinear terms in the energy balance, \cref{eq:Vdot-beam} holds for the linear problem as well, and so does \cref{eq:V-eps} with $R$ replaced by $0$. 

Now that we have verified that the control system \cref{eq:beam-v}-\cref{eq:torque} satisfies \cref{as:ISS,as:smooth} (and thus \cref{as:forwarding} as well by \cref{claim:sylvester}), we can
put it in cascade with the output integrator
\begin{equation}
\label{eq:int-theta-ref}
\dot{z} = \phi = \theta - \theta_{\mathrm{ref}}
\end{equation}
and finally use the forwarding feedback law \cref{eq:forwarding-feedback} to control the extended $[v, \phi, z]$-system. This is made possible by \cref{claim:sylvester}, which provides the (unique) solution $\M$ to \cref{eq:forwarding} and guarantees its required Lipschitz properties. Note again that $\M$ do not depend on the particular choice of $\theta_\mathrm{ref}$.
In order to apply \cref{claim:GAS,claim:LES}, it remains to check the non-resonance condition \cref{eq:non-resonance}. Here, this is done by computing the input-to-steady-state map: 
 constant input $u$ yields a stationary solution $[v, \phi]$ given by $v(\xi) = u \xi$, $\phi = u$. Therefore, by \cref{claim:LES,claim:GAS}, the controller \cref{eq:forwarding-feedback} achieves global asymptotic stabilization and local exponential stabilization in $X \times Y = H_1 \times H_0 \times \R$ of the nonlinear system \cref{eq:beam-v}-\cref{eq:torque} in cascade with \cref{eq:int-theta-ref}. Coming back to the original $[w, \theta]$-coordinates, we then see that the (unique) equilibrium at which $\theta = \theta_\mathrm{ref}$ is globally asymptotically stable and locally exponentially stable, and this holds for any choice of $\theta_\mathrm{ref}$. In summary, we have designed a dynamic feedback law for the original nonlinear plant \cref{eq:w-beam} that enables global set-point output tracking of the angular position $\theta$ at any reference $\theta_\mathrm{ref}$.

 \begin{rem}
In this analysis, we were able to derive global results in both initial data and reference by applying \cref{claim:LES,claim:GAS}
 after a suitable change of variables and control input instead of relying on the local \cref{th:set-point}. The underlying key property which we used here is that the original equations \cref{eq:w-beam}, although nonlinear, are left invariant under constant inputs -- of course, up to a change of variable related to the resulting steady state. \Cref{th:set-point} is applicable to systems lacking this feature but provides weaker results in comparison.
 \end{rem}

\section{Concluding remarks}

We have presented new results for stabilization of nonlinear systems consisting of a  semilinear component 
whose  output is integrated by a neutrally stable linear subsystem. For this purpose, we have introduced a new class of nonlinear Sylvester equations, which we demonstrated to be solvable. Our approach also provides a solution for the local set-point output tracking problem. As a case study, we have investigated the nonlinear dynamics of a flexible beam attached to a rigid body.

Now, while the formula \cref{eq:M-nonlin} completely determines the nonlinear map $\M$ and its Fr\'echet differential, and thus the state feedback \cref{eq:forwarding-feedback}, the exact implementation in practice seems out of reach in most cases. 
This is an issue even in the finite-dimensional context \cite{PraOrt01}. Nevertheless, 
in view of \cref{eq:M-nonlin},
our control based on the nonlinear map $\M$ can be seen as a perturbation of the linear solution $\M_0$ of \cref{eq:sylvester-lin}. Hence, the additional integral containing the nonlinear perturbation terms in \cref{eq:M-nonlin} can be interpreted as compensation terms that improve the control brought by $\M_0$. It is expected that taking into account even approximations of these compensation terms should lead better closed-loop performances, as it should be investigated in future works.

\appendix

\section{Additional material for the proofs of \cref{claim:sylvester,th:set-point}}
\label{sec:add-proof}

This section contains additional technical details and is intended for reading along \cite{VanBri23}.

\begin{proof}[Proof of \cref{claim:sylvester} (continued, sketch)]
It remains to prove that $\M$ is differentiable and $\d \M$ is locally Lipschitz continuous. First, the nonlinear semigroup $\{\T_t\}_{t \geq 0}$ is Fr\'echet differentiable \cite[Lemma 4.5]{VanBri23}. Furthermore, because $\d f$ is locally Lipschitz continuous with $\d f(0) = 0$, it follows from \cref{eq:ineq-first-var} that
there exists an open neighborhood $\mathcal{V}$ of $0$ in $X$ such that for all $x \in \mathcal{V}$ and $\delta \in \D(A)$,
\begin{equation}
\label{eq:new-first-var}
\langle A\delta + \d f(x)\delta, P \delta \rangle_X \leq - \frac{\mu}{2} \|\delta\|^2_X. 
\end{equation}
\Cref{eq:new-first-var} generalizes \cite[Hypothesis 3.3, item (ii)]{VanBri23}, where the same property had to hold globally in $x$ and with $P = \id$. It implies that, in the region $\mathcal{V}$, the nonlinear semigroup  $\{\T_t\}_{t \geq 0}$ is strictly contractive with respect to the (equivalent) norm $\| P^{1/2} \cdot \|_X$. This is enough for our purpose: by the property of semiglobal exponential stability, for any bounded set $\B$, there exists $T_\B > 0$ such that $\T_t \B \subset \mathcal{V}$ for all $t \geq T_\B$, and we then can obtain estimates of the form
\begin{equation}
\label{eq:exp-cont}
\| \T_tx_1 - \T_tx_2\|_X \leq K e^{- \tilde{\mu} t} \|x_1 - x_2\|_X
\end{equation}
holding for all $t \geq T_\B$ and $x_1, x_2 \in \B$, and also
\begin{equation}
\label{eq:exp-cont-dT}
\| [\d \T_t(x_1) - \d \T_t(x_2)] \delta \|_X \leq K e^{- \tilde{\mu} t} \|\delta\|_X \|x_1 - x_2 \|_X
\end{equation}
for all $t \geq T_\B$, $x_1, x_2 \in \B$ and $\delta \in X$, where the positive constants $K$ and $\tilde{\mu}$ can be chosen independent of $\B$. On the other hand, by the local Lipschitz properties of the nonlinear terms and the fact that there is a bounded set that contains all $\T_t \B$, $t \geq 0$, we can also obtain counterparts to \cref{eq:exp-cont,eq:exp-cont-dT} with exponential growth instead of decay and constants depending on $\B$, which is enough do deal with the finite time interval $[0, T_\B]$.
Therefore, one can adapt the arguments in the proof of \cite[Theorem 3.4]{VanBri23} by splitting in two the integral in \cref{eq:M-nonlin} with the time $T_\B$, where $\B$ is an arbitrary but fixed bounded set, ``differentiate under the integral sign'' in \cref{eq:M-nonlin} with Lebesgue's dominated convergence theorem to obtain differentiability,
and prove Lipschitz continuity of the differential by taking advantage of \cref{eq:exp-cont,eq:exp-cont-dT}. Note that \cite[Theorem 3.4]{VanBri23} deals with the case $S = 0$, 
but ``new'' terms stemming from the presence of $S$ in \cref{eq:ineq-first-var} are linear and do not change much to the proof.
\end{proof}

\begin{proof}[Proof of \cref{th:set-point} (sketch)]
At the beginning of \cite[Section 4.2.1]{VanBri23}, the proof of \cite[Theorem 3.2]{VanBri23} is outlined in five items. Let us describe the differences with \cref{th:set-point} itemise. We work in the new coordinate system $[x, \eta] \triangleq [x, z - \M(x)]$.
\begin{enumerate}[label=\arabic*.]
    \item This corresponds to \cref{eq:local-coerc}.
    \item A similar property can be obtained by taking advantage of \cref{eq:local-coerc,eq:new-first-var}.
    \item A counterpart to \cite[Lemma 4.2]{VanBri23} can be obtained by adapting our Lyapunov analysis from \crefrange{eq:Wb}{eq:Wb-strict} in presence of $\yr$.
    \item The same arguments are valid.
    \item This is not required here.
\end{enumerate}
\end{proof}


\bibliographystyle{abbrv}

\bibliography{forwarding-cdc-hal}

\end{document}